\documentclass[12pt]{article}
\usepackage[left=1in, right=1in, top=1in, bottom=1.1in]{geometry}
\usepackage{graphicx, color}
\usepackage{amssymb}
\usepackage{amsmath}
\usepackage{amsthm}
\usepackage{enumerate}

\newcommand{\NN}{\mathbf{N}}
\newcommand{\RR}{\mathbf{R}}
\newcommand{\II}{\mathbb{I}}
\newcommand{\JJ}{\mathbb{J}}

\newcommand{\rk}{\operatorname{rk}}
\newcommand{\covered}{< \hspace{-1.5ex} \cdot \hspace{1ex}}

\newcommand{\oo}{\text{\scalebox{1.5}{$\circ$}}}
\newcommand{\ox}{\text{\raisebox{.25ex}{$\otimes$}}}
\newcommand{\boxo}{\Box}
\newcommand{\boxx}{\boxtimes}
\newcommand{\V}{\mathrm{V}}
\newcommand{\w}{\operatorname{wt}}

\newcommand{\xto}{\xrightarrow}

\newcommand{\llangle}{\langle \! \langle}
\newcommand{\rrangle}{\rangle \! \rangle}

\newcommand{\trim}{\operatorname{trim}}

\newcommand{\layer}{\oplus}
\newcommand{\glue}{\oplus_G}
\newcommand{\stick}{\oplus_S}
\newcommand{\Top}{\operatorname{top}}
\newcommand{\Bot}{\operatorname{bot}}

\newtheorem{theorem}{Theorem}[section]

\newtheorem{cor}[theorem]{Corollary}
\newtheorem{prop}[theorem]{Proposition}

\theoremstyle{definition}

\newtheorem{defn}[theorem]{Definition}

\title{Enumeration of Graded $({\bf 3+1})$-Avoiding Posets}
\author{Joel Brewster Lewis and Yan X Zhang\\
Massachusetts Institute of Technology}

\begin{document}
\maketitle
\begin{abstract}
The notion of $({\bf 3+1})$-avoidance has shown up in many places in enumerative combinatorics, but the natural goal of enumerating all $({\bf 3+1})$-avoiding posets remains open. In this paper, we enumerate \emph{graded} $({\bf 3+1})$-avoiding posets for both reasonable definitions of the word ``graded.''  Our proof consists of a number of structural theorems followed by some generating function computations.  We also provide asymptotics for the growth rate of the number of graded $({\bf 3+1})$-avoiding posets. 

\end{abstract}

\section{Introduction}\label{sec:intro} 
The notion of $({\bf 3+1})$-avoiding posets appears in different different areas of combinatorics, such as in the Stanley-Stembridge conjecture about the $e$-positivity of certain chromatic symmetric functions \cite{stanley-stembridge} and the characterization of interval semiorders \cite{fishburn}.  Graph-theoretically, $({\bf 3+1})$-avoiding posets are exactly those posets whose comparability graphs are complements of claw-free graphs; as a result, they also are connected to a generalization of the ``birthday problem'' \cite{birthday}.

Despite these connections, the enumeration of $({\bf 3+1})$-avoiding posets has remained elusive.  This is particularly bothersome because the enumeration of posets that are both $({\bf 2+2})$- and $({\bf 3+1})$-avoiding, the interval semiorders, is well-understood: the number of unlabeled $n$-element interval semiorders is exactly the Catalan number $C_n$ \cite{fishburn}. Moreover, $({\bf 2+2})$-avoiding posets have recently been enumerated, as well \cite{2+2}.  Happily, there has been some progress: Skandera \cite{skandera:characterization} has given a characterization of all $({\bf 3+1})$-avoiding posets involving the square of the antiadjacency matrix and Atkinson, Sagan and Vatter \cite{sagan:3+1permutations} have recently characterized and enumerated $({\bf 3+1})$-avoiding permutations (i.e., permutations whose associated posets are $({\bf 3+1})$-avoiding).  

In this paper, we consider a related problem and enumerate \emph{graded} $({\bf 3+1})$-avoiding posets (for both common meanings of the word graded) via structural theorems and generating function computations. The property of gradedness is very natural and captures a lot of the complexity of the general case while making the problem much more tractable. We remark that a substantially easier problem is to enumerate $({\bf 3+1})$- and $({\bf 2+2})$-avoiding graded posets, and that the solution may be found in work of the second-named author currently in preparation; labeled $({\bf 3+1})$- and $({\bf 2+2})$-avoiding strongly graded posets are counted by the generating function $1 + \frac{e^x(e^x - 1)(e^x - 2)}{e^{2x} - e^x - 1}$. 

In the rest of this introduction, we summarize our strategy and results. In Section~\ref{sec:definitions}, we offer some definitions and notation that we will use throughout the paper. Then in Section~\ref{sec:local conditions}, we give a useful local condition that is equivalent to $({\bf 3+1})$-avoidance for graded posets.  

The main ideas of the paper are in Section~\ref{sec:simplifications}, where we introduce several operations that allow us to decompose strongly graded $({\bf 3+1})$-avoiding posets into simpler objects. First, in Section~\ref{sec:trimming} we reduce our problem of obtaining the generating function for all graded $({\bf 3+1})$-avoiding posets to studying certain posets we will call \emph{trimmed} which are slightly simpler but which capture most of the information of the original posets. Then, in Section~\ref{sec:layering}, we show that trimmed $({\bf 3+1})$-avoiding posets arise from taking ordinal sums of \emph{sum-indecomposable} $({\bf 3+1})$-avoiding posets. Finally, in Section~\ref{sec:sticking and gluing} we introduce two more operations, \emph{gluing} and \emph{sticking}.  We show that sum-indecomposable $({\bf 3+1})$-avoiding posets arise from gluing and sticking together basic units called \emph{quarks}, which we enumerate in Section~\ref{sec:gf for quarks}.

This line of argument culminates in Section~\ref{sec:climax}, in which we backtrack and use the results of the preceding sections and the transfer-matrix method to enumerate all strongly graded $({\bf 3+1})$-avoiding posets. We end with some extensions of these techniques. In Section~\ref{sec:with height} we use similar generating functional arguments to enumerate strongly graded $({\bf 3+1})$-avoiding posets by height.  We use this modified enumeration in Section~\ref{sec:weakly graded posets} to enumerate $({\bf 3+1})$-avoiding \emph{weakly} graded posets.  Finally, in Section~\ref{sec:asymptotics}, we use the generating functions computed in Sections~\ref{sec:climax} and~\ref{sec:weakly graded posets} to establish the asymptotic rate of growth of the number of graded $({\bf 3+1})$-avoiding posets.

An extended abstract of this work appeared as \cite{FPSAC}.

\emph{Note added in proof}: In recent work \cite{all3+1}, Guay-Paquet, Morales and Rowland have enumerated all $({\bf 3 + 1})$-avoiding posets using similar techniques.

\section{Preliminaries}\label{sec:definitions}

A \textbf{partially ordered set}, or \textbf{poset} for short, is a set with an irreflexive, antisymmetric and transitive relation~$>$.  We say two elements $a$, $b$ of a poset are \textbf{comparable} if $a \geq b$ or $b \leq a$. In this paper, we concern ourselves only with posets of finite cardinality. We say that an element $w$ \textbf{covers} an element $v$, denoted $v \covered w$, if $v < w$ and there is no $z$ such that $v < z < w$.  Observe that the order relations of a finite poset follow by transitivity from the cover relations; this allows us to graphically represent posets by showing only the cover relations.  The resulting graph is called the \textbf{Hasse diagram} of the poset. 

A poset in which every pair of elements is comparable is called a \textbf{chain}, and a poset in which every pair of elements is incomparable is called an \textbf{antichain}.  

We say that four elements $w$, $x$, $y$, $z$ in a poset $P$ are a \textbf{copy of ${\bf 3+1}$} if we have that $x < y < z$ and $w$ is incomparable to all of $x$, $y$, $z$.  If $P$ contains no copy of ${\bf 3+1}$, we say that $P$ \textbf{avoids} ${\bf 3+1}$. 

Call a poset $P$ \textbf{weakly graded} if there exists a rank function $\rk:P \to \NN$ such that if $a \covered b$ then $\rk(b) - \rk(a) = 1$ and such that the minimal occurring rank in each connected component is $0$.  Call a weakly graded poset \textbf{strongly graded} if all minimal elements have the same rank and all maximal elements have the same rank. (Equivalently, a poset is strongly graded if all maximal chains in the poset have the same length; in this case the rank function $\rk$ may be recovered by setting $\rk(v)$ to be the length of a longest chain whose maximal element is $v$.)  Figure \ref{fig:flavors of posets} gives examples of posets with these properties.  The \textbf{height} of a weakly graded poset $P$ is the number of vertices in a longest chain in $P$.

\begin{figure}
\centering
\includegraphics[scale=.5]{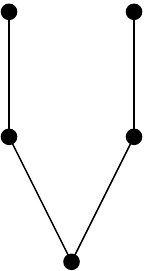}
\hspace{6ex}
\includegraphics[scale=.5]{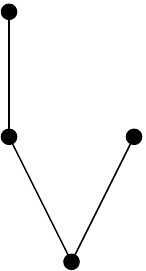}
\hspace{6ex}
\includegraphics[scale=.5]{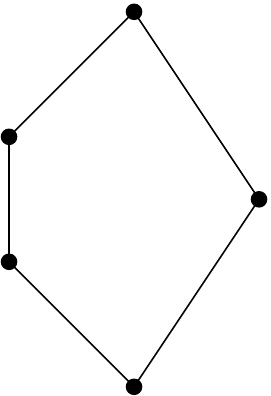}
\caption{Three posets: the first is strongly graded, the second is weakly graded but not strongly graded, and the third is not weakly graded.}
\label{fig:flavors of posets}
\end{figure}

A weakly graded poset $P$ of height $k + 1$ has \textbf{rank sets} $P(0)$, $P(1)$, \ldots, $P(k)$, where $P(i) = \{v\in P \mid \rk(v) = i\}$.  If $P$ is strongly graded, all the minimal elements are in $P(0)$ and all the maximal ones are in $P(k)$.  

Figure~\ref{fig:allsmall} shows all unlabeled weakly graded $({\bf 3+1})$-avoiding posets on four or fewer vertices.  Taking labelings into account, we see that for $n = 1, 2, 3$, and $4$ the number of weakly graded $({\bf 3+1})$-avoiding posets on $n$ vertices is $1$, $3$, $19$, and $195$, respectively.  Of these, respectively $1$, $3$, $13$ and $111$ are strongly graded.

\begin{figure}
\centering
\includegraphics[scale=.4]{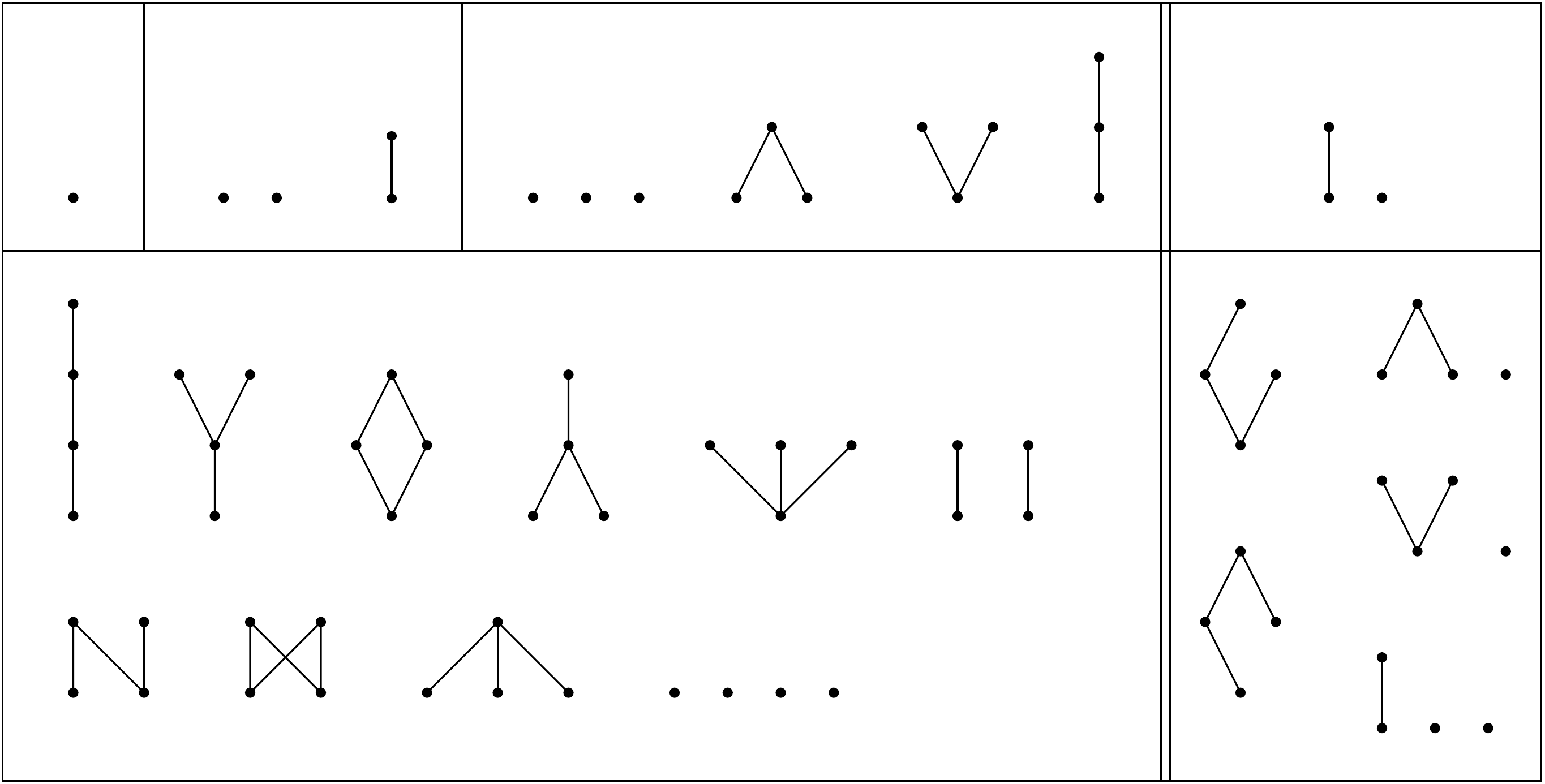}
\caption{All weakly graded $({\bf 3+1})$-avoiding posets on four or fewer vertices.  The doubled line separates the strongly graded posets (on the left) from the others.}
\label{fig:allsmall}
\end{figure}

In this paper, we avoid the use of the unmodified word ``graded'' because of an ambiguity in the literature: some sources (e.g., \cite{stanley:EC2}) use the word ``graded'' to mean ``strongly graded,'' while many others (e.g., \cite{klarner}) use ``graded'' to mean ``weakly graded.'' 

\section{Local Conditions}\label{sec:local conditions}

In this section, we give a concise local condition which is equivalent to $({\bf 3+1})$-avoidance for weakly graded posets.

Given a weakly graded poset $P$, call a vertex $v \in P$ of rank $i$ \textbf{up-seeing} if every vertex in $P(i+1)$ covers $v$. Similarly, call $v$ \textbf{down-seeing} if $v$ covers every vertex in $P(i-1)$. Let $\V(i)$ be the set of up-seeing vertices of rank $i$ and let $\Lambda(i)$ be the set of all down-seeing vertices of rank $i$.  (As a mnemonic, think of $v$ at the point of the $\V$ or $\Lambda$, with lots of edges going respectively up or down in the Hasse diagram of the poset.)  These definitions are illustrated in Figure \ref{fig:up and down seeing}.

\begin{figure}
\centering
\scalebox{.5}{\input{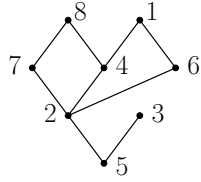}}
\caption{In the (weakly graded) poset pictured, the vertices labeled $1$, $8$, $4$, $2$ and $5$ are up-seeing and the vertices labeled $2$, $3$ and $5$ are down-seeing.  The vertices labeled $6$ and $7$ are neither up- nor down-seeing.}
\label{fig:up and down seeing}
\end{figure}

Call a vertex that is both up- and down-seeing an \textbf{all-seeing} vertex. Also, call a weakly graded poset $P$ \textbf{vigilant} if every vertex of $P$ is up-seeing, down-seeing, or all-seeing. One key consequence of our next result, Theorem~\ref{thm:3+1 locality v2}, is that in our study of graded $({\bf 3+1})$-avoiding posets it suffices to only consider vigilant posets.

\begin{theorem}
\label{thm:3+1 locality v2}
For a weakly graded poset $P$, the following are equivalent:
\begin{enumerate}[I.]
\item $P$ is $({\bf 3+1})$-avoiding;
\item $P$ is vigilant and every two vertices $v$, $w$ such that $\rk(w) - \rk(v) \geq 2$ are comparable;
\item $P$ is vigilant and every two vertices $v$, $w$ such that $\rk(w) - \rk(v) = 2$ are comparable.
\end{enumerate}
\end{theorem}

\begin{proof}
It is clear that \textit{II} implies \textit{III}, so we will show that \textit{III} implies \textit{II}, that \textit{II} implies \textit{I}, and that \textit{I} implies \textit{II}. 

$\textit{III} \Rightarrow \textit{II}$: Let $P$ be a poset satisfying the conditions in \textit{III}.  We show that every two vertices whose ranks differ by $3$ are comparable; the result follows by induction.  Choose vertices $v$ of rank $i$ and $w$ of rank $i + 3$.  Since there is a vertex of rank $i + 3$, there must be at least one vertex $z$ of rank $i + 1$, and by \textit{III} we have $w > z$.  Since $P$ is graded, there is some vertex $y$ of rank $i + 2$ such that $w > y > z$.  But also $y > v$ by \textit{III}, so $w > v$, as desired.

$\textit{II} \Rightarrow \textit{I}$: Let $P$ be a poset satisfying the conditions in \textit{II}; we show $P$ avoids ${\bf 3+1}$.  Consider any $3$-chain $x < y < z$ in $P$ and any other vertex $w \in P$. We claim that $w$ is comparable to at least one of $x$, $y$, $z$.  By the defining properties of $P$, if $\rk(w) < \rk(z) - 1$ then $w < z$ while if $\rk(w) > \rk(x) + 1$ then $w > x$, and in either case we have our result.  The only remaining case is $\rk(z) - 1 = \rk(w) = \rk(x) + 1$.  In this case, since $w$ is either up- or down-seeing, we conclude that $w$ is comparable to at least one of $x$ and $z$.  Thus, $P$ avoids ${\bf 3+1}$, as desired.

$\textit{I} \Rightarrow \textit{II}$: Let $P$ be a weakly graded $({\bf 3 + 1})$-avoiding poset.  First, we show that two vertices whose ranks differ by $2$ or more are comparable. Choose vertices $u$ and $w$ at ranks $i$ and $j$ respectively with $j - i \geq 2$.  Since there are vertices at ranks at least $i + 2$, there must be a chain $x \covered y \covered z$ with $x$ at rank $i$.  Because $P$ avoids $\bf 3+1$, $u$ must be comparable to at least one of these vertices and so in particular $u < z$.  Then there is a chain $u \covered v \covered z$ in $P$, and by $({\bf 3+1})$-avoidance we have that $w$ is comparable to some member of this chain and so finally $w > v$ as desired.

Second, we show that $P$ is vigilant.  Suppose for contradiction that we have a vertex $v$ of rank $i$ that is neither up- nor down-seeing. This means $v \not \in \Lambda(i) \cup \V(i)$.  Then there exist vertices $u$, $w$ such that $\rk(u) = \rk(v) - 1$, $\rk(w) = \rk(v) + 1$, and $v$ is incomparable to both $u$ and $w$.  But by the preceding paragraph, $u < w$, and so there is some vertex $v'$ of rank $i$ such that $u < v' < w$.  This chain together with $v$ is a copy of ${\bf 3+1}$ in $P$.  This is a contradiction, so $P$ is vigilant. 
\end{proof}

We introduce the following convention for representing vigilant posets: vertices that are all-seeing are represented by squares, vertices that are up-seeing are represented by downwards-pointing triangles, and vertices that are down-seeing are represented by upwards-pointing triangles.  (Thus, each vertex has horizontal edges on the sides on which it is connected to all vertices.)  This convention is illustrated in Figure~\ref{fig:vertexconvention}.

\begin{figure}
\centering
\includegraphics[scale=.5]{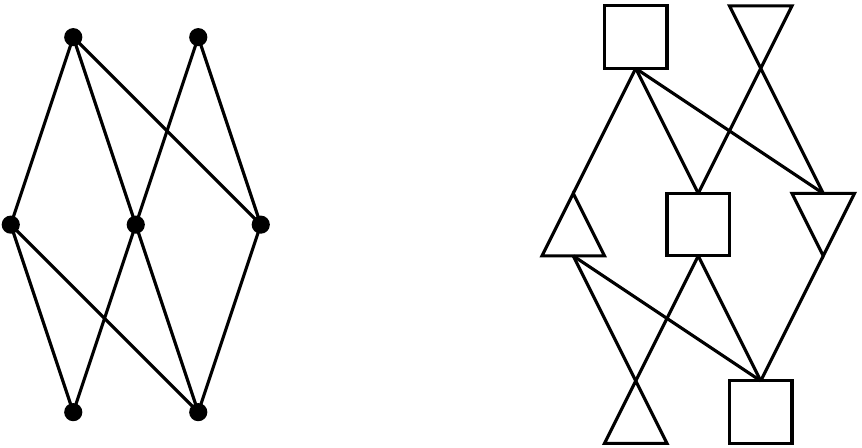}
\caption{The Hasse diagram for the vigilant poset at left will be displayed as the image at right: all-seeing vertices are represented as squares, other vertices as triangles.}
\label{fig:vertexconvention}
\end{figure}

\section{Simplifications}\label{sec:simplifications}

In this section, we introduce four operations that allow us to count vigilant posets by working instead with simpler objects.  We show that $({\bf 3+1})$-avoidance will be mostly compatible with these simplifications, reducing the problem of enumerating graded $({\bf 3+1})$-avoiding posets basically to studying vigilant posets of height $2$.  In Section~\ref{sec:trimming} we work with weakly graded posets, while in Sections~\ref{sec:layering} and~\ref{sec:sticking and gluing} we restrict ourselves to strongly graded posets.  (We will return to weakly graded posets in Section~\ref{sec:weakly graded posets}.)

\subsection{Trimming}\label{sec:trimming}

We call a vigilant poset $P$ \textbf{trimmed} if it has the following properties:
\begin{itemize}
\item every rank has at most one all-seeing vertex,
\item the all-seeing vertices are unlabeled, and
\item the other $m$ vertices are labeled with $[m]$.
\end{itemize}

\begin{figure}
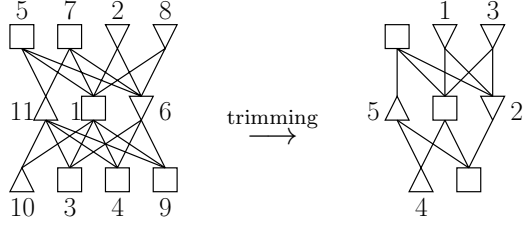

\centering
\scalebox{.5}{\input{untrimmed3+1avoidingtex.tex}} \hspace{2ex} \raisebox{.4in}{$\overset{\text{trimming}}{\longrightarrow}$} \hspace{2ex} \scalebox{.5}{\input{trimmed3+1avoidingtex.tex}}
\caption{A strongly graded $({\bf 3+1})$-avoiding poset and the associated trimmed poset.}
\label{fig:trimming}
\end{figure}

Given a weakly graded poset $P$, there is a naturally associated trimmed poset, denoted $\trim(P)$, that we get by removing the all-seeing vertices from $P$, adding a single unlabeled all-seeing vertex to every rank set from which we removed all-seeing vertices, and relabeling the other vertices so as to preserve the relative order of labels.  Figure~\ref{fig:trimming} provides one illustration of this operation.

\begin{prop}
\label{prop:3+1 behavior trimming}
The weakly graded vigilant poset $P$ avoids ${\bf 3+1}$ if and only if $\trim(P)$ does.
\end{prop}
\begin{proof}
It is routine to check that neither of the conditions of Theorem~\ref{thm:3+1 locality v2}(\textit{III}) is affected by the trimming map.
\end{proof}

Since we lose very little information when we replace the poset $P$ by the trimmed poset $\trim(P)$, Proposition~\ref{prop:3+1 behavior trimming} suggests that we can reduce the enumeration of labeled graded $({\bf 3+1})$-avoiding posets to the enumeration of trimmed $({\bf 3+1})$-avoiding posets.  The following proposition makes this intuition precise.  Let $w_n$ be the number of weakly graded $({\bf 3+1})$-avoiding posets on $n$ vertices and let
\[
W(x) = \sum_n w_n \frac{x^n}{n!}
\] be the exponential generating function for labeled weakly graded $({\bf 3+1})$-avoiding posets.  
\begin{prop}
\label{prop:trimming gf}
The exponential generating function for labeled weakly graded $({\bf 3+1})$-avoiding posets is
\[
W(x) = \sum_{n, r} a_{n, r} \frac{x^n}{n!} (e^x - 1)^r.
\]
where $a_{n, r}$ is the number of trimmed $({\bf 3+1})$-avoiding posets with $r$ all-seeing vertices and $n$ other vertices.

An analogous result holds if we restrict attention to the strongly graded posets.
\end{prop}
\begin{proof}
A weakly graded $({\bf 3+1})$-avoiding poset $P$ is uniquely determined by the associated trimmed poset $T$ and the set of labels for the all-seeing vertices at each rank.  Moreover, any trimmed $({\bf 3+1})$-avoiding poset $T$ with all-seeing vertices at $r$ levels together with an appropriate tuple of $r$ nonempty sets of labels yields a weakly graded $({\bf 3+1})$-avoiding poset.  Thus, by standard rules for generating functions (or equivalently from species-type considerations as in \cite[Section 4]{jackson-moffatt-morales}), the generating function for weakly graded $({\bf 3+1})$-avoiding posets with all-seeing vertices at exactly $r$ ranks is $(e^x - 1)^r \cdot \sum_n a_{n, r} \frac{x^n}{n!}$.  Summing over $r$ gives the result.
\end{proof}

\subsection{Ordinal sums}\label{sec:layering}

Suppose we have two trimmed strongly graded posets $P_1$ and $P_2$ of heights $a$ and $b$, respectively. We can take the \textbf{ordinal sum} of $P_1$ and $P_2$ by letting the lowest-ranked elements in $P_2$ cover all highest-ranked elements in $P_1$ and relabeling in a way consistent with the labelings of $P_1$ and $P_2$.  (Thus, there are many ways to take ordinal sums of $P_1$ and $P_2$; all the resulting posets are isomorphic under relabeling.)  We denote one of the possible resulting posets of height $a+b$ by $P_1 \layer P_2$.  See for example Figure~\ref{fig:layering}.  In the context of vigilant posets, the ordinal sum is an especially nice operation because a vertex in $P_1$ or $P_2$ which is up-seeing and/or down-seeing retains that property in $P_1 \layer P_2$.

\begin{figure}
\centering
\scalebox{.5}{\input{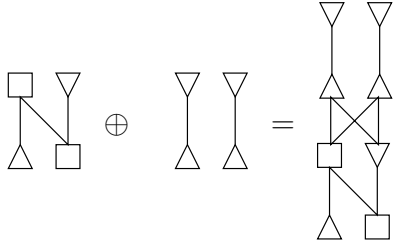}}
\caption{An ordinal sum of two sum-indecomposable posets.  (Labels are suppressed for readability.)}
\label{fig:layering}
\end{figure}

Call a nonempty strongly graded poset $P$ with height $k$ \textbf{sum-indecomposable} if $P$ is trimmed and there is no $i < k - 1$ for which every vertex in $P(i)$ is up-seeing (equivalently, there is no $i > 0$ for which every vertex in $P(i)$ is down-seeing). This word choice is motivated by the existence of a decomposition of trimmed posets into sum-indecomposables.

\begin{prop}
\label{prop:weak decomposition}
A trimmed strongly graded poset $P$ can be written uniquely as
\[
P = P_1 \layer P_2 \layer \cdots \layer P_m,
\]
for a sequence $(P_1, P_2, \ldots, P_m)$ of sum-indecomposable posets. 
\end{prop}
\begin{proof}
Let $P$ be a trimmed strongly graded poset and let $k$ be the height of $P$.  Take the smallest rank $i$ for which $P(i)$ has all up-seeing vertices. If $i=k$, then $P$ is sum-indecomposable. Otherwise, we can write $P = P_1 \layer P'$, where $P_1$ has height $i+1$ and is sum-indecomposable by the minimality of $i$. Repeating this process gives us the desired sequence, which is obviously unique.
\end{proof}

\begin{prop}
\label{prop:3+1 behavior layering}
If a trimmed strongly graded poset $P$ decomposes into sum-indecomposable posets as $P = P_1 \layer \cdots \layer P_m$, then $P$ avoids ${\bf 3+1}$ if and only if all of the $P_i$ avoid ${\bf 3+1}$.
\end{prop}
\begin{proof}
One direction is trivial: if any of the $P_i$ contains a copy of ${\bf 3+1}$ then certainly $P$ does as well. For the other direction, suppose that $P$ contains a copy of ${\bf 3+1}$ with vertices $u<v<w$ and $x$. If $x \in P_i$, then $u$, $v$, and $w$ must also all be in $P_i$ (since two vertices from different $P_j$ must be comparable), so $P_i$ itself does not avoid ${\bf 3+1}$.
\end{proof}

Propositions \ref{prop:weak decomposition} and \ref{prop:3+1 behavior layering} simplify the problem of counting strongly graded $({\bf 3+1})$-avoiding posets: it now suffices to count sum-indecomposable posets and then count the ways to combine them by ordinal sum.  As we will see in Proposition~\ref{prop:final layering}, this is a simple task with generating functions.  Thus, we now turn our attention to enumerating sum-indecomposable $({\bf 3+1})$-avoiding posets.

\subsection{Sticking and Gluing}\label{sec:sticking and gluing}

In order to enumerate sum-indecomposable posets, we break them down into more manageable pieces that we call \emph{quarks}.    We show that quarks can be combined to make posets using two operations that we call \emph{sticking} and \emph{gluing}, that every sum-indecomposable poset can be written uniquely as a sticking and gluing of quarks, and that $({\bf 3+1})$-avoidance is encoded nicely in this decomposition.

Observe that every poset of height $1$ or $2$ is weakly graded and so naturally has a rank function that assigns all minimal vertices to rank $0$ and all other vertices to rank $1$.  A \textbf{quark} $Q$ is a pair $(P, r)$ of a poset $P$ and function $r: P \to \{0, 1\}$, with the following restrictions:
\begin{itemize}
\item $P$ has height $1$ or $2$ (and so consequently is weakly graded);
\item $P$ does not have both an up-seeing vertex at rank $0$ and a down-seeing vertex at rank $1$;
\item if $v \in P$ is not isolated then $r(v) = \rk(v)$;
\item there exist vertices $v_1$ and $v_2$ in $P$ such that $r(v_1) = 0$ and $r(v_2) = 1$.
\end{itemize}
Equivalently, thinking in terms of Hasse diagrams, one may view a quark as a bipartite graph with a designated bipartition of the vertices into nonempty lower and upper halves with the restriction that at most one part of the bipartition contains an all-seeing vertex.  

Given a vertex $v \in Q$, we say that $r(v)$ is the \textbf{rank} of $v$.  Note that two different quarks may have the same underlying poset, and that the underlying poset of a quark may have height $1$ even though the quark itself has two nonempty ranks. 

Given the close relationship between quarks and posets, we extend our poset terminology to this new context in the natural way.  Notably, the adjectives ``vigilant'' and ``trimmed'' have the same meaning for quarks as for posets, and since every quark is vigilant we use the same convention for displaying their vertices as was introduced for vigilant posets in Section~\ref{sec:local conditions}.

\begin{figure}
\begin{center}
\includegraphics[scale=.5]{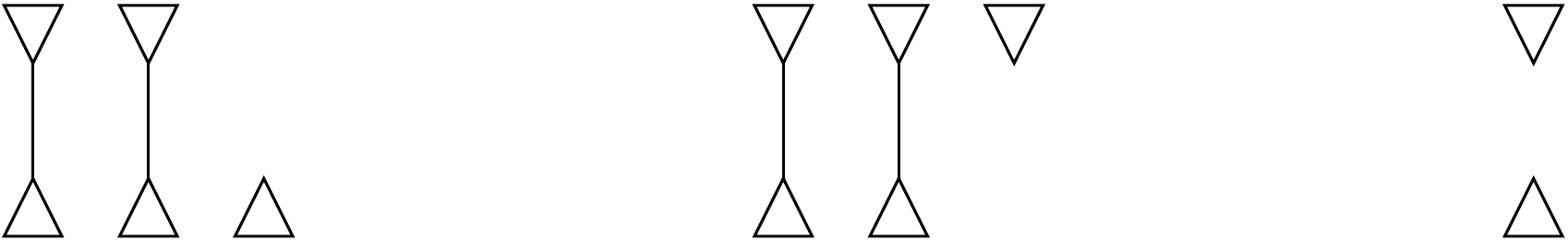}
\end{center}
\caption{Three quarks.  All three quarks are middle quarks; the left quark is also a bottom quark, while the center quark is also a top quark.  The left and center quarks have isomorphic underlying posets but are not isomorphic as quarks.  The right quark has underlying poset of height $1$ (with two vertices and no relations), while as a quark it has vertices at two different rank sets.}
\label{fig:quarks}
\end{figure}

There are three classes of quarks that will be of interest two us; see Figure~\ref{fig:quarks} for examples.  A \textbf{bottom quark} is a trimmed quark in which every isolated vertex (if any) is assigned to rank $0$ and there are no all-seeing vertices of rank $1$.  A \textbf{middle quark} is a quark with no all-seeing vertices of either rank.  A \textbf{top quark} is a trimmed quark in which every isolated vertex (if any) is assigned to rank $1$ and there are no all-seeing vertices of rank $0$.  Observe that every trimmed quark belongs to at least one of these classes and that many quarks belong to more than one of them.

We now introduce the two operations that can be used to build every sum-indecomposable poset of height larger than $2$ from quarks.  We first describe the operations for height-$3$ posets, and afterwards the general case.

Given a bottom quark $Q_0$ and a top quark $Q_1$, we say that a poset $P$ arises from \textbf{sticking} $Q_0$ and $Q_1$ if the following conditions hold:
\begin{itemize}
\item The vertex set of $P$ is the disjoint union of the vertex sets of $Q_0$ and $Q_1$.
\item For $i = 0, 1$, if $v, w \in Q_i$, then $v < w$ in $P$ if and only if $v < w$ in $Q_i$.
\item For $j = 0, 1$, if $v \in Q_0$ and $w \in Q_1$ have rank $j$ in their respective quarks then $v < w$ in $P$.
\item The only other order relations of $P$ are those that follow by transitivity.
\item The labeling of vertices of $P$ is consistent with the labelings of $Q_0$ and $Q_1$.
\end{itemize}
In an abuse of notation, we denote this relationship by $P = Q_0 \stick Q_1$.  Similarly, we say that $P$ arises from \textbf{gluing} $Q_0$ and $Q_1$, and we write $P = Q_0 \glue Q_1$, if the following conditions hold:
\begin{itemize}
\item $P$ has a (not necessarily induced) subposet $P' = Q_0 \stick Q_1$.
\item The rank set $P(1)$ has an additional (unlabeled) all-seeing vertex, the additional order relations implied by the presence of this vertex and transitivity, and no other order relations.
\end{itemize}

It is easy to check that posets of the form $Q_0 \stick Q_1$ and $Q_0 \glue Q_1$ are sum-indecomposable posets of height $3$.  Also observe that, as in the case of ordinal sums, a vertex in $Q_0$ or $Q_1$ that is up-seeing or down-seeing keeps this status after either gluing or sticking.  Figure~\ref{fig:stick and glue 2 quarks} shows an example of the sticking and gluing of two quarks.

\begin{figure}
\centering
\scalebox{.5}{\input{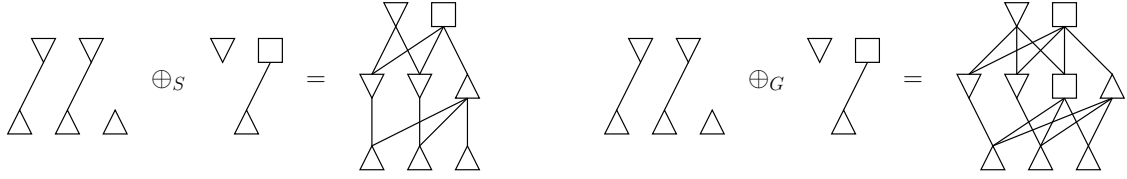}}
\caption{Sticking and gluing a bottom quark and a top quark to build sum-indecomposable posets.  (Labels are suppressed for readability.)}
\label{fig:stick and glue 2 quarks}
\end{figure}

Now we describe how to apply these operations to many quarks in order to create posets of larger height.  
\begin{defn}
Suppose we are given a bottom quark $Q_0$, middle quarks $Q_1$, \ldots, $Q_{k - 1}$, and a top quark $Q_k$.  For each choice $(\alpha_1, \ldots, \alpha_{k}) \in \{S, G\}^{k}$, we say that a trimmed poset $P$ is of the form $Q_0 \oplus_{\alpha_1} Q_1 \oplus_{\alpha_2} \cdots \oplus_{\alpha_{k}} Q_k$ if the following conditions hold:
\begin{itemize}
\item For $i \in \{0, \ldots, k + 1\}$, the $i$th rank set $P(i)$ consists of the disjoint union of $Q_i(0)$ and $Q_{i - 1}(1)$ and, if $\alpha_i = G$, an unlabeled all-seeing vertex.
\item For $i \in \{0, \ldots, k\}$, if $v, w \in Q_i$ then $v < w$ in $P$ if and only if $v \covered w$ in $Q_i$.
\item For $i \in \{0, \ldots, k - 1\}$ and $j \in \{0, 1\}$, if $v \in Q_i(j)$ and $w \in Q_{i + 1}(j)$ then $v \covered w$.
\item For $i \in \{0, \ldots, k - 2\}$, if $v \in Q_i(1)$ and $w \in Q_{i + 2}(0)$ then $v \covered w$.
\item All other order relations of $P$ follow by transitivity from those of the four preceding bullet points.
\item The labeling of vertices of $P$ is consistent with the labelings of the $Q_i$.
\end{itemize}
\end{defn}

As before, we denote this relation by $P = Q_0 \oplus_{\alpha_1} \cdots \oplus_{\alpha_k} Q_k$.  An example of a poset of height $4$ formed by sticking and gluing is shown in Figure~\ref{fig:stick and glue 3 quarks}.

\begin{figure}
\centering
\scalebox{.5}{\input{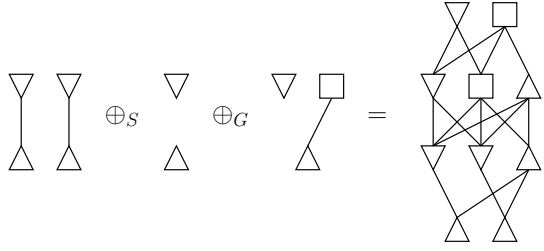}}
\caption{An example of sticking and gluing quarks to build a sum-indecomposable poset.  (Labels are suppressed for readability.)}
\label{fig:stick and glue 3 quarks}
\end{figure}

The quark decomposition is useful because it behaves nicely with resprect to up-seeing and down-seeing vertices.

\begin{prop}\label{prop:seeing preserved}
Suppose that $Q_0$ is a bottom quark, $Q_1$, \ldots, $Q_{k - 1}$ are middle quarks and $Q_k$ is a top quark, and 
\[
P = Q_0 \oplus_{\alpha_1} Q_1 \oplus_{\alpha_2} \cdots \oplus_{\alpha_{k}} Q_k.
\]
A vertex $v \in Q_i$ is up-seeing (respectively, down-seeing) in $Q_i$ if and only if it is up-seeing (respectively, down-seeing) in $P$. 
\end{prop}
\begin{proof}
Choose $i \in \{0, \ldots, k\}$ and choose $v \in Q_i$.  If $v$ is not down-seeing in $Q_i$, then $v \in Q_i(1)$ and there is some $w \in Q_i(0)$ such that $v \not > w$.  In this case, $v \in P(i + 1)$ is not larger than $w \in P(i)$, and so $v$ is not down-seeing in $P$.  On the other hand, if $v \in Q_i$ is down-seeing in $Q_i$ then by construction $v$ covers all vertices of one lower rank in $P$.
\end{proof}

\begin{prop}
Suppose that $Q_0$ is a bottom quark, $Q_1$, \ldots, $Q_{k - 1}$ are middle quarks and $Q_k$ is a top quark, and 
\[
P = Q_0 \oplus_{\alpha_1} Q_1 \oplus_{\alpha_2} \cdots \oplus_{\alpha_{k}} Q_k.
\]
We have that $P$ is sum-indecomposable.
\end{prop}
\begin{proof}
It follows immediately from the construction that $P$ is weakly graded, vigilant and trimmed.  Observe that the restrictions on the quarks guarantee that every vertex not of maximal rank is covered by something and that every vertex not of minimal rank covers something, so that $P$ is strongly graded as well.  Finally, for all $i \in \{0, \ldots, k\}$, $Q_i(0)$ contains a vertex that is not up-seeing.  Thus, by Proposition~\ref{prop:seeing preserved}, $P(i)$ contains such a vertex, and thus $P$ is sum-indecomposable.
\end{proof}

In fact, every sum-indecomposable poset may be written as a sticking and gluing of quarks in a unique way, as the next result shows.

\begin{prop}\label{prop:strong decomposition for sum-indecomposables}
For $k \geq 1$, suppose that $P$ is a sum-indecomposable poset of height $k + 2$.  There exists a unique bottom quark $Q_0$, top quark $Q_k$, collection $Q_1$, \ldots, $Q_{k - 1}$ of middle quarks, and choice $(\alpha_1, \ldots, \alpha_k) \in \{S, G\}^k$ such that 
\[
P = Q_0 \oplus_{\alpha_1} Q_1 \oplus_{\alpha_2} \cdots \oplus_{\alpha_{k}} Q_k.
\]
\end{prop}
\begin{proof}
For some $k \geq 1$, choose a sum-indecomposable poset $P$ of height $k + 2$.  We first describe the decomposition of $P$ into quarks, then show that it is unique.

For $i \in \{1, \ldots, k - 1\}$, define $Q_i$ as follows: the lower rank set $Q_i(0)$ consists of all vertices of $P(i)$ that are not up-seeing, the upper rank set $Q_i(1)$ consists of all vertices of $P(i + 1)$ that are not down-seeing, and the vertices are labeled in accordance with the labeling of $P$.  The top and bottom quark $Q_0$ and $Q_k$ are defined similarly, except that $Q_0(0) = P(0)$ and $Q_k(1) = P(k + 1)$ (i.e., we remove the additional restriction in this case).  For $i \in \{1, \ldots, k\}$, we set $\alpha_i = G$ if $P(i)$ contains an all-seeing vertex, and $\alpha_i = S$ otherwise.  

Since $P$ is sum-indecomposable, both rank sets of every $Q_i$ are nonempty, and no $Q_i$ contains an all-seeing vertex except possibly in the bottom rank set of $Q_0$ or the top rank set of $Q_k$.  Since $P$ is strongly graded, every vertex in $P(1)$ covers some vertex in $P(0)$, so $Q_0(1)$ has no isolated vertices (and likewise $Q_k(0)$ has no isolated vertices).  Thus, $Q_0$ is a bottom quark, $Q_k$ a top quark, and $Q_1$, \ldots, $Q_{k - 1}$ are middle quarks.  Since $P$ is trimmed, every vertex of $P$ either belongs to exactly one of the $Q_i$ or is an all-seeing vertex not of top or bottom rank.  Finally, it's easy to check that the cover relations of $P$ and those of $Q_0 \oplus_{\alpha_1} Q_1 \oplus_{\alpha_2} \cdots \oplus_{\alpha_{k}} Q_k$ are the same, as desired.

The uniqueness of this decomposition is straightforward: the presence of all-seeing vertices indicates which of the $\alpha_i$ are $G$, vertices of rank $i$ that are not down-seeing can only come from $Q_{i - 1}(1)$, and vertices of rank $i$ that are not up-seeing can only come from $Q_i(0)$.  Thus, the partition of the underlying set of $P$ into the underlying sets of the $Q_i$ is uniquely determined; the uniqueness of the $Q_i$ as quarks follows immediately.
\end{proof}

Now we can connect our characterization of sum-indecomposable posets as quarks that have been glued or stuck together to our ultimate goal of studying $({\bf 3+1})$-avoiding posets.

\begin{theorem}
\label{thm:3+1 characterization}
A sum-indecomposable poset $P$ is $({\bf 3+1})$-avoiding if and only if the decomposition $P = Q_0 \oplus_{\alpha_1} Q_1 \oplus_{\alpha_2} \cdots \oplus_{\alpha_k} Q_k$ into quarks satisfies the following condition: for every occurrence of $Q_i \stick Q_{i+1}$ in the decomposition, either $Q_i$ has no isolated vertices on its bottom rank or $Q_{i+1}$ has no isolated vertices on its top rank, or both.
\end{theorem}
\begin{proof}
The poset $P$ avoids $\bf 3+1$ if and only if the conditions of Theorem~\ref{thm:3+1 locality v2}(\textit{III}) hold.  The first condition holds for every sum-indecomposable poset by definition, so the desired statement reduces to the claim that $P$ contains two incomparable vertices whose ranks differ by $2$ if and only if there is some $i \geq 0$ such that $\alpha_{i + 1} = S$, $Q_i$ has an isolated vertex on its lower rank set and $Q_{i + 1}$ has an isolated vertex on its upper rank set.  

The vertices $u$ and $w$ are comparable as elements of $P$ if and only if there is some $v \in P(i + 1)$ such that $u < v < w$.  If $u \not\in Q_{i}(0)$ then we can take $v$ to be any vertex in $Q_i(1)$, while if $w \not\in Q_{i + 1}(1)$ then we can take $v$ to be any vertex in $Q_{i + 1}(0)$, so in these cases $u$ and $w$ are always comparable.  Now we consider the case that $u \in Q_i(0)$ and $w \in Q_{i + 1}(1)$.   If $\alpha_{i + 1} = G$ then we can take $v$ to be the all-seeing vertex at rank $i + 1$, so in this case $u$ and $w$ are comparable.  If $\alpha_{i + 1} = S$ and $u$ and $w$ are not both isolated in their respective quarks, we may assume without loss of generality that there is some $v \in Q_i(1)$ such that $u < v$.  By the sticking construction, $v < w$, and so $u$ and $w$ are comparable in this case as well.  Finally, suppose that $\alpha_{i + 1} = S$ and that $u$ and $w$ are both isolated in their respective quarks.  We wish to show that $u$ and $w$ are incomparable.  Since $\alpha_{i + 1} = S$, we have $P(i + 1) = Q_i(1) \cup Q_{i + 1}(0)$.  Thus, for any $v \in P(i + 1)$ we have that $v$ is incomparable with $u$ or with $w$.  

It follows that $P$ contains two isolated vertices whose rank differs by $2$, and so a copy of ${\bf 3+1}$, if and only if there is some $i$ such that $\alpha_{i + 1} = S$, $Q_i$ has an isolated vertex of rank $0$, and $Q_{i + 1}$ has an isolated vertex of rank $1$, as desired.
\end{proof}

With this result in hand, we now turn to the task of counting quarks.

\section{Quarks}\label{sec:gf for quarks}

Theorem~\ref{thm:3+1 characterization} implies that studying sum-indecomposable $({\bf 3+1})$-avoiding posets reduces to studying quarks.  In this section, we set out to enumerate quarks.  Following the observation at the beginning of the previous section, this amounts to enumerating bipartite graphs with certain restrictions: a quark $Q$ with $m$ vertices in $Q(0)$ and $n$ vertices in $Q(1)$ is, up to differences in the labeling scheme, just a particular kind of bipartite graph on the disjoint union $[m] \uplus [n]$.  We enumerate such graphs, keeping track of some simple structural information about them.

We define a family of sets $A_\mu^\nu(m,n)$, where $\mu$ and $\nu$ are subsets (possibly empty) of $\{\boxo, \oo, \boxx, \ox \}$, as follows:
\begin{itemize}
\item $A_\mu^\nu(m,n)$ is the set of bipartite graphs on $[m] \uplus [n]$ with some restrictions. The elements of $\nu$ correspond to restrictions on the vertices in $[n]$ and the elements of $\mu$ correspond to restrictions on the vertices of $[m]$.  (Here the placement of indices is meant to suggest that vertices in $[m]$ form a bottom rank and the vertices in $[n]$ a top rank.)  An empty set of symbols corresponds to no restrictions on the corresponding set.
\item A $\boxo$ corresponds to the requirement that there be at least one all-seeing vertex; a $\boxx$ corresponds to the requirement that there be no all-seeing vertex.  
\item A $\oo$ corresponds to the requirement that there be an isolated vertex; a $\ox$ corresponds to the requirement that there be no isolated vertex.
\end{itemize}

For example, $A(m, n)$ is the set of all bipartite graphs on $[m] \uplus [n]$ and $A^\boxo_\boxx (m,n)$ is the subset of $A(m, n)$ containing those graphs with at least one all-seeing vertex in $[n]$ but no all-seeing vertices in $[m]$. 

Of these sets, we are particularly interested in those that contain quarks.  The top quarks correspond to the graphs in $A_{\boxx\ox}$, the bottom quarks correspond to the graphs in $A^{\boxx\ox}$, and the middle quarks correspond to the graphs in $A_\boxx^\boxx$.  In the next section, we will need to consider a more refined count of middle quarks; thus, for $\nu, \mu \subset \{\oo, \ox\}$ we define $B_\mu^\nu(m,n) = A_{\{\boxx\} \cup \mu}^{\{\boxx\} \cup \nu}(m,n)$.  For example, $B_\oo^\ox(m, n)$ is the set of bipartite graphs on $[m] \uplus [n]$ with no all-seeing vertices, no isolated vertices in $[n]$, and at least one isolated vertex in $[m]$. For each $\mu$ and $\nu$, let 
\begin{equation}\label{eqn:def of F}
F^\nu_\mu(x) = \sum_{m, n \geq 1} |B^\nu_\mu (m,n)| \frac{x^{m+n}}{m! n!}
\end{equation}
be the corresponding generating function, so for example the coefficient of $\frac{x^N}{N!}$ in $F^\ox_\oo(x)$ is the number of middle quarks on $N$ vertices with at least one isolated vertex of rank $0$ but none of rank $1$.  Finally, let $B^\nu_\mu$ be the union over $m$ and $n$ of all $B^\nu_\mu(m,n)$. Note that the set of middle quarks is a disjoint union 
\[
B = B^\oo_\oo \cup B^\oo_\ox \cup B^\ox_\oo \cup B^\ox_\ox,
\]
which manifests as a sum of formal power series
\[
F = F^\oo_\oo + F^\oo_\ox + F^\ox_\oo + F^\ox_\ox.
\]

\begin{prop}
\label{prop:quark gfs}
Let 
\[
\Psi(x) = \sum_{m,n \geq 0} \frac{2^{mn}x^{m+n}}{m!n!}
\]
and let $F^\nu_\mu$ be defined as in Equation~\eqref{eqn:def of F}.  We have
\begin{align*}
F^\oo_\oo(x) & = (1 - e^{-x})^2 \Psi(x), \\
F^\oo_\ox(x) = F^\ox_\oo(x) & = (1 - e^{-x})((2e^{-x}-1) \Psi(x) - 1), \\
\intertext{and}
F^\ox_\ox(x) & = (2e^{-x} - 1)((2e^{-x} - 1) \Psi(x) - 1).
\end{align*}
\end{prop}
\begin{proof}
See Appendix~\ref{section:yan calculus}.
\end{proof}

\section{Strongly Graded Posets}\label{sec:climax}

In this section, we use the $F^\nu_\mu$ as building blocks to obtain the generating function for sum-indecomposable $({\bf 3+1})$-avoiding posets, and then proceed to enumerate all strongly graded $({\bf 3+1})$-avoiding posets.  (Recall that by our definition, sum-indecomposable posets are necessarily strongly graded.)  We begin by encoding a sum-indecomposable poset in terms of a \emph{word} that keeps track of its quarks and how they are combined (i.e., gluing and sticking).  Then we use the transfer-matrix method to enumerate words while keeping track of the restrictions imposed by Theorem~\ref{thm:3+1 characterization}. 

Given any quark $Q$, we define its \textbf{type} as follows: if $Q$ is a middle quark (i.e., an element of $B$), then the type of $Q$ is the symbol $B^\nu_\mu$ corresponding to the unique subset among the four $B^\nu_\mu$ to which it belongs.  (This is a slight abuse of notation that will never cause ambiguity in context.)  If $Q$ is a top or bottom quark, first remove any all-seeing vertices from $Q$, leaving a middle quark $Q'$; then set the type of $Q$ to be the type of $Q'$.  Define a \textbf{word} to be any monomial in the noncommutative algebra $\RR\llangle S,G,B^\oo_\oo, B^\oo_\ox, B^\ox_\oo, B^\ox_\ox \rrangle$.  We now encode the properties of sum-indecomposability and $({\bf 3+1})$-avoidance into conditions on words.

\begin{defn}
\label{defn:legal}
We say that a word $L$ is \textbf{legal} if for some $k \geq 0$ there are $\alpha_i \in \{S, G\}$ and $B_i \in \{B^\oo_\oo, B^\oo_\ox, B^\ox_\oo, B^\ox_\ox\}$ such that $L = \alpha_0 B_0 \alpha_1 B_1 \alpha_2 \cdots B_{k-1}\alpha_{k}B_k\alpha_{k+1}$, and none of the following occur:
\begin{enumerate}
\item $\alpha_0 = S$ and $B_0$ has a $\oo$ in the superscript;
\item $\alpha_{k+1} = S$ and $B_k$ has a $\oo$ in the subscript;
\item there is some $i$, $1 \leq i \leq k$, such that $B_{i-1}$ has a $\oo$ in the subscript, $\alpha_i = S$, and $B_{i}$ has a $\oo$ in the superscript.
\end{enumerate}
\end{defn}

We define a \textbf{weight function} $\w : \RR\llangle S,G,B^\oo_\oo, B^\oo_\ox, B^\ox_\oo, B^\ox_\ox \rrangle \to \RR[\![x, z]\!]$ as follows: we set $\w(S) = 1$, $\w(G) = z$, and $\w(B^\nu_\mu) = F^\nu_\mu$ and we extend by linearity and multiplication.

Let $I_{\geq 2}(x, z)$ be the generating function for sum-indecomposable $({\bf 3+1})$-avoiding posets of height at least $2$, where the variable $z$ counts all-seeing vertices, the variable $x$ counts other vertices, and $I_{\geq 2}(x, z)$ is exponential in $x$ and ordinary in $z$.

\begin{theorem}
\label{thm:sum over words most heights}
 The generating function for sum-indecomposable $({\bf 3+1})$-avoiding posets of height at least $2$ is
\[
I_{\geq 2}(x,z) = \sum_L \w(L),
\]
where the sum is over all legal words $L$.
\end{theorem}
\begin{proof}
First we handle the height-$2$ case.  By considering the presence or absence of an all-seeing vertex of rank $0$ or $1$, it's easy to see that the generating function for such posets is precisely
\[
z^2 \cdot F + z\cdot F^\ox + z \cdot F_\ox + F^\ox_\ox
\]
and that this is equal to the sum of $\w(L)$ over all legal words $L$ of length $3$ (i.e., those of the form $\alpha_0 B_0 \alpha_1$ satisfying certain restrictions).  Now we handle the case of larger heights.

Let $P$ be a sum-indecomposable $({\bf 3+1})$-avoiding poset of height $k + 2$ for some $k \geq 1$.  Suppose $P$ decomposes into quarks as $P = Q_0 \oplus_{\alpha_1} \cdots \oplus_{\alpha_{k}} Q_k$.  Set $W(P)$ to be the word $\alpha_0 B_0 \alpha_1 B_1 \alpha_2 \cdots B_{k-1} \alpha_{k} B_k \alpha_{k+1}$ defined as follows: 
\begin{itemize}
\item for $0 \leq i \leq k + 1$, $\alpha_i = G$ if $P$ has an all-seeing vertex of rank $i$ and $\alpha_i = S$ otherwise;
\item for $0 \leq i \leq k$, $B_i$ is the type of $Q_i$.
\end{itemize}
It is easy to check that the map $W$ is well-defined and that the constraints imposed on the $Q_i$ and $\alpha_i$ by Theorem~\ref{thm:3+1 characterization} correspond precisely to the condition that $W(P)$ is a legal word.  Given a legal word $L$, we now show that the generating function for posets $P$ such that $W(P) = L$ is precisely $\w(L)$; our result follows immediately from summing over all legal words $L$.

Fix a word $L = \alpha_0 B_0 \alpha_1 \cdots B_{k-1} \alpha_{k} B_k \alpha_{k+1}$, and consider its preimage $W^{-1}(L) = \{P \mid W(P) = L\}$.  Any $P \in W^{-1}(L)$ can be written in the form $P = Q_0 \oplus_{\alpha_1} \cdots \oplus_{\alpha_{k}} Q_k$ with the types of the $Q_i$ determined by the $B_i$. However, after we fix the type $B_i$, any quark of that type can be used as part of a sum-indecomposable $({\bf 3+1})$-avoiding poset. Thus, the posets in the preimage of $L$ contribute exactly $F^\nu_\mu$ for each occurrence of $B_i = B^\nu_\mu$. Furthermore, each occurrence of $\alpha_i = G$ corresponds to a single all-seeing vertex, and so contributes $z$.  Thus, by standard rules for generating functions, the generating function for posets in $W^{-1}(L)$ is exactly $\w(L)$.  It follows that $I_{\geq 2}(x, z)$ is the result of summing $\w(L)$ over the legal words $L$, as desired.
\end{proof}

Let $I(x, z)$ be the generating function for nonempty sum-indecomposable $({\bf 3+1})$-avoiding posets, where the variable $z$ counts all-seeing vertices, the variable $x$ counts other vertices, and $I(x, z)$ is exponential in $x$ and ordinary in $z$.

\begin{cor}
\label{cor:sum over words all heights}
The generating function for all nonempty sum-indecomposable $({\bf 3+1})$-avoiding posets is 
\[
I(x,z) = z + \sum_L \w(L),
\] 
where the sum is over all legal words $L$.  
\end{cor}

The preceding results establish that to enumerate posets we may focus our energies on enumerating words.  We accomplish this task with the transfer-matrix method.  Let $M_W$ be the matrix
\[
M_W = G \cdot \begin{bmatrix}
B^\oo_\oo & B^\oo_\ox & B^\ox_\oo & B^\ox_\ox \\
B^\oo_\oo & B^\oo_\ox & B^\ox_\oo & B^\ox_\ox \\
B^\oo_\oo & B^\oo_\ox & B^\ox_\oo & B^\ox_\ox \\
B^\oo_\oo & B^\oo_\ox & B^\ox_\oo & B^\ox_\ox
\end{bmatrix} + S \cdot \begin{bmatrix}
0        & 0       & B^\ox_\oo & B^\ox_\ox \\
B^\oo_\oo & B^\oo_\ox & B^\ox_\oo & B^\ox_\ox \\
0        & 0       & B^\ox_\oo & B^\ox_\ox \\
B^\oo_\oo & B^\oo_\ox & B^\ox_\oo & B^\ox_\ox
\end{bmatrix}
\]
with entries in the noncommutative algebra $\RR\llangle S,G,B^\oo_\oo, B^\oo_\ox, B^\ox_\oo, B^\ox_\ox \rrangle$ of words. 

\begin{prop}
\label{prop:gf words}
With $M_W$ as above, the sum of the legal words of length $2k+3$ is
\[
\begin{bmatrix}G \cdot B^\oo_\oo &G \cdot B^\oo_\ox&(S + G) B^\ox_\oo&(S+G)B^\ox_\ox\end{bmatrix} \cdot (M_W)^{k} \cdot \begin{bmatrix}G \\ S + G \\ G \\ S+G \end{bmatrix}.
\] 
\end{prop}
\begin{proof}
Consider the graph $G_w$ with vertices $\{*,B^\oo_\oo, B^\oo_\ox, B^\ox_\oo, B^\ox_\ox\}$ and the following directed, labeled edges: for every pair $u, v$ of vertices (allowing $u = v$), $G_w$ has a directed edge $u \xto{G} v$, and $G_w$ has a directed edge $u \xto{S} v$ unless $u = B^\nu_\oo$ or $u = *$ and $v = B^\oo_\mu$ or $v = *$.  The graph $G_w$ is illustrated in Figure~\ref{fig:graph}.

\begin{figure}
\centering
\scalebox{.6}{\input{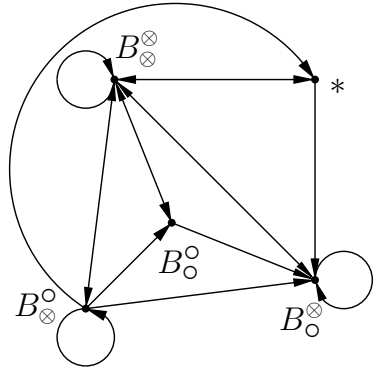}}
\caption{The $S$-labeled edges of the graph $G_w$ defined in the proof of Proposition~\ref{prop:gf words}.  Each pair of vertices is also joined by directed edges labeled $G$ (not shown).}
\label{fig:graph}
\end{figure}

We identify each walk
\[
* \xto{\alpha_0} B_0 \xto{\alpha_1} \cdots \xto{\alpha_k} B_k \xto{\alpha_{k+1}} *
\]
with the word
\[
\alpha_0 B_0 \alpha_1 \cdots B_k \alpha_{k+1}.
\]
Observe that the first two conditions in Definition~\ref{defn:legal} correspond to the restrictions on edges involving $*$ and the final condition corresponds to edges not involving $*$.  Thus the legal words are exactly the walks on this graph that start and end at $*$, with no intermediate copies of $*$.  We enumerate these walks using the transfer-matrix method, as in \cite[Section 4.7]{stanley:EC1}.

Let 
\[
X = \begin{bmatrix}
B^\oo_\oo & 0 & 0 & 0 \\
0 & B^\oo_\ox & 0 & 0 \\
0 & 0 & B^\ox_\oo & 0 \\
0 & 0 & 0 & B^\ox_\ox
\end{bmatrix}
\]
 and 
\[Y = G \cdot \JJ + S \cdot \begin{bmatrix}
0 & 0 & 1 & 1 \\
1 & 1 & 1 & 1 \\
0 & 0 & 1 & 1 \\
1 & 1 & 1 & 1
\end{bmatrix},\]
where $\JJ$ is the $4 \times 4$ matrix whose entries are all equal to $1$.  Examining $G_w$ and applying \cite[Theorem 4.7.1]{stanley:EC1}, we have that the sum of the words associated to the aforementioned walks is
\[
\begin{bmatrix}G & G & S + G & S + G \end{bmatrix}
X Y X Y \cdots X  
\begin{bmatrix}G \\ S + G \\ G \\ S+G \end{bmatrix},
\]
which is equivalent to the desired expression. 
\end{proof}

Let $M$ be the matrix 
\[
M = z \cdot \begin{bmatrix}
F^\oo_\oo & F^\oo_\ox & F^\ox_\oo & F^\ox_\ox \\
F^\oo_\oo & F^\oo_\ox & F^\ox_\oo & F^\ox_\ox \\
F^\oo_\oo & F^\oo_\ox & F^\ox_\oo & F^\ox_\ox \\
F^\oo_\oo & F^\oo_\ox & F^\ox_\oo & F^\ox_\ox
\end{bmatrix} + \begin{bmatrix}
0       & 0        & F^\ox_\oo & F^\ox_\ox \\
F^\oo_\oo & F^\oo_\ox & F^\ox_\oo & F^\ox_\ox \\
0       & 0        & F^\ox_\oo & F^\ox_\ox \\
F^\oo_\oo & F^\oo_\ox & F^\ox_\oo & F^\ox_\ox
\end{bmatrix}
\]
(whose entries are in $\RR[\![x,z]\!]$).  

\begin{cor}\label{cor:matrix magic}
The generating function for all sum-indecomposable $({\bf 3+1})$-avoiding posets of height at least $2$ is
\[
I_{\geq 2}(x,z) = \begin{bmatrix}zF^\oo_\oo &zF^\oo_\ox&(1 + z)F^\ox_\oo&(1+z)F^\ox_\ox\end{bmatrix} \cdot (\II - M)^{-1} \cdot \begin{bmatrix}z \\ 1 + z \\ z \\ 1+z \end{bmatrix},
\]
where $\II$ is the $4 \times 4$ identity matrix and $M$ is as above.
\end{cor}
\begin{proof}
The result follows from Theorem~\ref{thm:sum over words most heights}, Proposition~\ref{prop:gf words} and the fact that the weight map $\w$ is an algebra homomorphism between $\RR\llangle S,G,B^\oo_\oo, B^\oo_\ox, B^\ox_\oo, B^\ox_\ox\rrangle$ and $\RR[\![x, z]\!]$.
\end{proof}

Now that we have enumerated sum-indecomposable $({\bf 3+1})$-avoiding posets, the only remaining step is to express the generating function for all $({\bf 3+1})$-avoiding posets in terms of the generating function for sum-indecomposables.  This turns out to be extremely simple.  

\begin{prop}
\label{prop:final layering}
The generating function $G_T(x, z)$ for all trimmed strongly graded $({\bf 3+1})$-avoiding posets is given by
\[
G_T(x, z) = (1 - I(x,z))^{-1}.
\]
(Recall that $I(x, z)$ is the generating function for nonempty sum-indecomposable posets.)
\end{prop}
\begin{proof}
By Proposition~\ref{prop:weak decomposition} and Proposition~\ref{prop:3+1 behavior layering} each trimmed strongly graded $({\bf 3+1})$-avoiding poset $P$ corresponds to a unique ordinal sum $P_1 \layer P_2 \layer \cdots \layer P_k$ of sum-indecomposable $({\bf 3+1})$-avoiding posets, and all such sequences give a trimmed strongly graded $({\bf 3+1})$-avoiding poset $P$.  The result follows from the compositional formula for generating functions.
\end{proof}

The only thing remaining is arithmetic.
\begin{theorem}
\label{thm:strong final answer}
The generating function for all strongly graded $({\bf 3+1})$-avoiding posets is
\[
1 + \frac{e^{2x}(2e^x - 3) + e^x(e^x - 2)^2 \Psi(x)}{e^x(2e^x + 1) + (e^{2x} - 2e^x - 1) \Psi(x)}.
\]
\end{theorem}
\begin{proof}
This is just a calculation, combining Corollaries~\ref{cor:sum over words all heights} and~\ref{cor:matrix magic} with Propositions~\ref{prop:trimming gf}, \ref{prop:quark gfs} and~\ref{prop:final layering}.  For $\#P = 0$, $1$, \ldots, the resulting number of posets is $1$, $1$, $3$, $13$, $111$, $1381$, $22383$, \ldots.
\end{proof}

\section{Strongly Graded Posets Counted by Height}\label{sec:with height}

In this section, we refine the generating function of the previous section to count strongly graded $({\bf 3+1})$-avoiding posets with $n$ vertices of height $k$.  (This refinement is a natural one to ask for on its own terms; it will also be of use to us when we enumerate weakly graded $({\bf 3+1})$-avoiding posets in the next section.)  The only change in our approach is that we keep track of the height of the poset as we glue and stick quarks, and then again as we take the ordinal sum of sum-indecomposables.  To this end, let $b_{n, k}$ be the number of strongly graded $({\bf 3+1})$-avoiding posets on $n$ vertices of height $k$ and let 
\begin{equation}\label{eqn:strong with height}
H(x, t) = \sum_{n, k} b_{n, k} \frac{x^n}{n!}t^k
\end{equation}
be the generating function for these numbers.  To compute $H(x, t)$, we return to the ideas of Section~\ref{sec:climax}.  Using the same method as in the proof of Corollary~\ref{cor:matrix magic} but keeping track of height, we have that the generating function for sum-indecomposable $({\bf 3+1})$-avoiding posets of height $2$ or more is 
\[
H_{I}(x, z, t) = t^2 \begin{bmatrix}zF^\oo_\oo &zF^\oo_\ox&(1+z)F^\ox_\oo&(1+z)F^\ox_\ox\end{bmatrix} \cdot (\II - t \cdot M)^{-1} \cdot \begin{bmatrix}z \\ 1+z \\ z \\ 1+z \end{bmatrix}.
\]
If we let $H_T(x, z, t)$ be the generating function for trimmed strongly graded $({\bf 3+1})$-avoiding posets, then (as in Proposition~\ref{prop:final layering}) we have that
\[
H_T(x, z, t) = (1 - tz - H_{I}(x, z, t))^{-1},
\]
while from the same reasoning as in the proof of Proposition~\ref{prop:trimming gf} we have $H(x, t) = H_T(x, e^x - 1, t)$.  Working out the arithmetic gives the following result.
\begin{prop}\label{prop:posets by height}
Let $H(x, t)$ be the generating function counting strongly graded $({\bf 3+1})$-avoiding posets by number of vertices and height (as in Equation \eqref{eqn:strong with height}).  We have
\[
H(x, t) = \frac{e^x(e^x + te^{2x} + t^2(e^x - 1)^2) + t((1 - 3e^x + e^{2x}) + t(e^x - 1)^2(e^x - 2)) \Psi(x)}{e^x(e^x + t e^x +t^2) + ((1 - 3e^x + e^{2x})t + (e^x - 2)t^2) \Psi(x)}.
\]
\end{prop}

The resulting coefficients are shown in Table~\ref{table:strong by height}.  

\begin{table}
\[
\begin{array}{p{5ex}l|*{7}{c}}
\multicolumn{9}{c}{\textrm{Height}} \\[1.5ex]
       &   & 0 & 1 &  2  &  3  &  4 &  5 & 6 \\
\cline{2-9}
       & 0 & 1 &   &     &     &    &    &   \\
       & 1 &   & 1 &     &     &    &    &   \\
       & 2 &   & 1 &   2 &     &    &    &   \\
$\# P$ & 3 &   & 1 &   6 &   6 &    &    &   \\
       & 4 &   & 1 &  50 &  36 &  24&    &   \\
       & 5 &   & 1 & 510 & 510 & 240& 120&   \\
       & 6 &   & 1 &7682 &7380 &4800&1800& 720
\end{array}
\]
\caption{The number of strongly graded $({\bf 3+1})$-avoiding posets of six or fewer vertices, by height.}
\label{table:strong by height}
\end{table}

\section{Weakly Graded Posets}\label{sec:weakly graded posets}

In this section, we expand our study to weakly graded posets.  We seek to apply the same methods that worked in the strongly graded case.  The results of Section~\ref{sec:local conditions}, the definition of a trimmed poset, and the results of Section~\ref{sec:trimming} carry over immediately to weakly graded posets.  We now seek to extend the rest of our work to this context.

We begin by proving Proposition~\ref{prop:not much difference}, which shows that weakly graded $({\bf 3+1})$-avoiding posets mostly ``look just like'' strongly graded posets. 
\begin{prop}\label{prop:not much difference}
In a weakly graded $({\bf 3+1})$-avoiding poset of height $k + 1$ such that $k \geq 1$, all maximal elements are of rank $k$ or $k - 1$ and all minimal elements are of rank $0$ or $1$.
\end{prop}
\begin{proof} 
Let $P$ be a weakly graded $({\bf 3+1})$-avoiding poset of height at least $3$.  A maximal vertex in $P$ is precisely the same as a vertex not comparable to vertices of any higher rank.  However, by Theorem \ref{thm:3+1 locality v2}, every vertex of $P$ is comparable to all vertices of rank two larger.  Putting these two facts together, we immediately conclude that $P$ has no vertices of rank two larger than any of its maximal vertices; this is the desired result. The case of minimal vertices is identical.
\end{proof}

With this result in hand, we can immediately extend the remaining results of Section~\ref{sec:simplifications} to weakly graded posets.  
\begin{cor}\label{cor:all the same}
Every trimmed $({\bf 3+1})$-avoiding poset $P$ can be decomposed uniquely as an ordinal sum of sum-indecomposable posets and each of the resulting sum-indecomposable posets can be decomposed uniquely by sticking and gluing quarks (including a top and bottom quark), where these objects and operations are defined as before, subject to the following changes: 
\begin{itemize}
\item when $P$ is written as a maximal ordinal sum of nonempty posets, the topmost and bottommost summands may be weakly graded, not necessarily strongly graded; and
\item when $P$ is written as an ordinal sum and the topmost summand is written as a sticking and gluing of quarks, the top quark may have isolated vertices on its lower rank set; similarly, the bottom quark of the bottommost summand may have isolated vertices on its upper rank set.  (These isolated vertices are exactly the maximal vertices not of maximum rank and the minimal vertices not of minimum rank, respectively.)  In other words, the topmost and bottommost quarks in these summands may be middle quarks.
\end{itemize}
\end{cor}  
This result allows us to directly apply the methods of Section~\ref{sec:climax}.

Observe that any $({\bf 3+1})$-avoiding poset with a chain containing $3$ or more elements must be connected, while posets of height $2$ can be somewhat more ``wild.''  This suggests that we should consider separately posets of height $2$ or less and posets of height $3$ or more.  We do this in the following sections.

\subsection{Posets of height at most $2$}\label{sec:short}

All posets of height at most $2$ are weakly graded and avoid ${\bf 3+1}$.  Of these, there is exactly one with height $0$ (the empty poset), and for each $n \geq 1$ there is exactly one poset on $n$ vertices of height $1$ (the antichain on $n$ vertices).  The number of posets of height $2$ on $n$ vertices is precisely $\sum_{m=1}^{n-1} \binom{n}{m} |A^\ox(n-m, m)|$: we choose $m$ vertices to be at rank $1$, and none of these can be isolated.  It follows from these three cases and from Appendix~\ref{section:yan calculus} that the generating function for weakly graded $({\bf 3+1})$-avoiding posets of height at most $2$ is
\[
1 + t (e^x - 1) + t^2 (e^{-x} \Psi(x) - e^x).
\]

\subsection{Posets of height at least $3$}

It follows from Corollary~\ref{cor:all the same} that when we write a trimmed $({\bf 3+1})$-avoiding poset as a maximal ordinal sum of smaller posets, the middle summands (if any) are all strongly graded, and so are sum-indecomposable under the definition of Section~\ref{sec:layering}.  The bottom summand satisfies the same rules for gluing and sticking quarks as before, except that the bottommost quark may be a middle quark rather than a bottom quark, and similarly for the top summand.  Equivalently, we may redefine \emph{legal word} by removing conditions 1 and 2 in Definition~\ref{defn:legal} to allow words that begin with $S B^\oo_\oo$ or $S B_\ox^\oo$ and end with $B^\oo_\oo S$ or $B_\oo^\ox S$.  This corresponds to a straightforward change in the generating function computations of Section~\ref{sec:climax}: the matrices $M_W$ and $M$ that appear in Proposition~\ref{prop:gf words} and Corollary~\ref{cor:matrix magic} do not need to change at all, though the vectors by which we multiply on the left and right need to be adjusted.  We must only take care in proving the analogue of Proposition~\ref{prop:final layering} in this context.

We now give a detailed plan of action.  We handle separately those posets that can and cannot be a summand in a nontrivial ordinal sum.  This gives us two cases: 
\begin{itemize}
\item posets of height $k + 1$, where $k \geq 2$, that consist of a single sum-indecomposable layer with at least one minimal vertex of rank $1$ and at least one maximal vertex of rank $k - 1$, and
\item posets of height $k + 1$, where $k \geq 2$, that do not fall into the previous class; these posets have no minimal vertices above rank $0$, or have no maximal vertices below rank $k$, or decompose as a nontrivial ordinal sum.
\end{itemize}
We compute the generating functions in these cases following the transfer-matrix approach used previously.  Note one important subtlety: in both cases, the transfer-matrix method generates some posets of height $2$ or less which we view as spurious.  Thus, we use the refined version of the generating functions computed in Section \ref{sec:with height} and make sure to eliminate the height-$0$, $1$ and $2$ terms in the first two cases. (The reason for this approach is that the transfer-matrix method as applied here fundamentally works on quarks; thus, it counts posets with isolated vertices multiple times (once for every possible assignment of the isolated vertices to rank $0$ or rank $1$).  Strongly graded posets of height $2$ or larger have no isolated vertices and so this issue does not arise in the strongly graded case.)

\subsubsection{Sum-indecomposable posets that cannot be used in an ordinal sum}\label{sec:can't be layered}

Some sum-indecomposable $({\bf 3+1})$-avoiding posets of height $k + 1$ cannot be used in a nontrivial ordinal sum to make another $({\bf 3+1})$-avoiding poset; these are exactly the ones with maximal vertices of rank $k - 1$ and minimal vertices of rank $1$.  Following the line of argument that culminated in Corollary~\ref{cor:matrix magic}, we see that the generating function for these posets is precisely
\[
t^2 \cdot \begin{bmatrix} F^\oo_\oo & F^\oo_\ox & 0 & 0 \end{bmatrix} \cdot (\II - t M)^{-1} \cdot \begin{bmatrix} 1 \\ 0 \\ 1 \\ 0 \end{bmatrix}
\]
where
\[
M = z \cdot \begin{bmatrix}
F^\oo_\oo & F^\oo_\ox & F^\ox_\oo & F^\ox_\ox \\
F^\oo_\oo & F^\oo_\ox & F^\ox_\oo & F^\ox_\ox \\
F^\oo_\oo & F^\oo_\ox & F^\ox_\oo & F^\ox_\ox \\
F^\oo_\oo & F^\oo_\ox & F^\ox_\oo & F^\ox_\ox
\end{bmatrix} + \begin{bmatrix}
0       & 0        & F^\ox_\oo & F^\ox_\ox \\
F^\oo_\oo & F^\oo_\ox & F^\ox_\oo & F^\ox_\ox \\
0       & 0        & F^\ox_\oo & F^\ox_\ox \\
F^\oo_\oo & F^\oo_\ox & F^\ox_\oo & F^\ox_\ox
\end{bmatrix}
\]
as before.

\subsubsection{All other posets}\label{sec:can be layered}

Trimmed $({\bf 3+1})$-avoiding posets not counted in the previous cases have height at least $3$ and can be written as (possibly trivial) ordinal sums of the following form: 
\begin{itemize}
\item they may or may not have a bottom layer (i.e., a sum-indecomposable ordinal summand) with all maximal vertices of the same rank but with minimal vertices at ranks $0$ and $1$;
\item they have some number (possibly $0$) of ``middle layers'' that are strongly graded sum-indecomposable posets; and
\item they may or may not have a top layer with all minimal vertices at rank $0$ but with maximal vertices at the rank below maximum rank.
\end{itemize}
The generating function for strongly graded sum-indecomposable $({\bf 3+1})$-avoiding posets is the function $H_I(x, z, t)$ defined in Section~\ref{sec:with height}.  We define $\Top(x, z, t)$ to be the generating function for sum-indecomposable $({\bf 3+1})$-avoiding posets with all minimal vertices of rank $0$ and with some maximal vertices of non-maximum rank, and analogously we define the generating function $\Bot(x, z, t)$.  Then the generating function for posets in this class coincides with
\begin{equation}\label{eqn:weak penultimate}
(1 + \Top(x, z, t)) (1 - H_I(x, z, t))^{-1} (1 + \Bot(x, z, t))
\end{equation}
for all powers of $t$ greater than or equal to $3$.  Moreover, we have
\[
\Top(x, z, t) = t^2 \cdot \begin{bmatrix}zF^\oo_\oo &zF^\oo_\ox&(1+z)F^\ox_\oo&(1+z)F^\ox_\ox\end{bmatrix} \cdot (\II - t M)^{-1} \cdot \begin{bmatrix} 1 \\ 0 \\ 1 \\ 0 \end{bmatrix}
\]
and 
\[
\Bot(x, z, t) = t^2 \cdot \begin{bmatrix} F^\oo_\oo & F^\oo_\ox & 0 & 0\end{bmatrix} \cdot (\II - t M)^{-1} \cdot \begin{bmatrix} z \\ 1 + z \\ z \\ 1 + z \end{bmatrix}.
\]

\subsection{All weakly graded $({\bf 3+1})$-avoiding posets}

Finally, we combine the work in the preceding subsections to enumerate weakly graded $({\bf 3+1})$-avoiding posets.

\begin{theorem}\label{thm:weak gf}
The generating function for weakly graded $({\bf 3+1})$-avoiding posets counted by number of vertices and height is
\begin{multline*}
1 + (e^{x} - 1)t + (e^{-x} \Psi(x) - e^x) t^2 + \\
t^3 \frac{e^{3x} + e^{3x}t - e^x(2 e^x +(1 + 2e^x - e^{2x})t) \Psi(x) - ((1 - 3e^x + e^{2x}) + (e^x - 2)t)\Psi(x)^2}{e^x(e^x + t e^x + t^2) + ((1 - 3e^x + e^{2x})t + (e^x - 2)t^2)\Psi(x)}.
\end{multline*}
The generating function for weakly graded $({\bf 3+1})$-avoiding posets counted by number of vertices is
\[
(e^{-x} - 1)\Psi(x) + \frac{2e^{3x} + e^{2x}(e^{x} - 2)\Psi(x)}{e^x(2e^{x} + 1) + (e^{2x} - 2e^x - 1) \Psi(x)}.
\]
\end{theorem}
\begin{proof}
The proof is a straightforward (albeit messy) computation: we add the generating functions from Section~\ref{sec:can't be layered} to the expression from Equation \eqref{eqn:weak penultimate}, kill the $t^0$, $t^1$ and $t^2$ terms, and add the result to the generating function from Section~\ref{sec:short}.  Substituting $t = 1$ and rearranging slightly gives the second formula. 
\end{proof}

The resulting coefficients are shown in Table~\ref{table:weak by height}.  Disregarding height, the numbers of weakly graded $({\bf 3+1})$-avoiding posets on $0$, $1$, $2$, \ldots vertices are $1$, $1$, $3$, $19$, $195$, $2551$, $41343$, \ldots.

\begin{table}
\[
\begin{array}{p{5ex}l|ccccccc}
\multicolumn{9}{c}{\textrm{Height}} \\[1.5ex] 
       &   & 0 & 1 &  2  &  3  & 4  & 5  & 6 \\
\cline{2-9}
       & 0 & 1 &   &     &     &    &    &   \\
       & 1 &   & 1 &     &     &    &    &   \\
       & 2 &   & 1 &  2  &     &    &    &   \\
$\# P$ & 3 &   & 1 & 12  &  6  &    &    &   \\
       & 4 &   & 1 & 86  &  84 &  24&    &   \\
       & 5 &   & 1 &840  &1110 & 480& 120&   \\
       & 6 &   & 1 &11642&16620&9120&3240& 720
\end{array}
\]
\caption{The number of weakly graded $({\bf 3+1})$-avoiding posets of six or fewer vertices, by height.}
\label{table:weak by height}
\end{table} 

\section{Asymptotics}\label{sec:asymptotics}

In this section, we compute asymptotics for the number of graded $({\bf 3+1})$-avoiding posets.  First, we give asymptotics for the coefficients of the series $\Psi(x)$; then we give the asymptotics for our posets in terms of the asymptotics of $\Psi$. Define $\displaystyle \psi_n = \sum_{i = 0}^n \frac{2^{i(n - i)}}{i! (n - i)!}$ so that $\Psi(x) = 1 + \sum_{n \geq 1} \psi_n x^n$.  In the next result, we give asymptotics for the coefficients $\psi_n$. 

\begin{prop}\label{prop:asymptotics for psi}              
There exist constants $C_1$ and $C_2$ such that           
\[     
\psi_{2k} \sim C_1 \cdot \frac{2^{k^2}}{(k!)^2} \qquad \textrm{and} \qquad \psi_{2k + 1} \sim C_2 \cdot \frac{2^{k(k+1)}}{k! (k+1)!}.
\]     
\end{prop}            
\begin{proof}         
For $n = 2k$ even, we can write            
\begin{align*}        
\psi_{2k} & = \sum_{i = 0}^{2k} \frac{2^{i(2k - i)}}{i! (2k - i)!} \\          
& = \sum_{i = -k}^k \frac{2^{k^2 - i^2}}{(k + i)!(k - i)!} \\
& \sim \frac{2^{k^2}}{(k!)^2} \vartheta_3(0, 1/2)
\end{align*}          
where $\vartheta_3$ is a Jacobi theta function.  Similarly, when $n = 2k + 1$ is odd we find 
\[
\psi_{2k+1} \sim 2^{1/4} \vartheta_2(0, 1/2) \frac{2^{k(k+1)}}{k!(k+1)!},
\]
as needed.
\end{proof}     
 
Let $g_n$ be the number of strongly graded $({\bf 3+1})$-avoiding posets on $n$ vertices and let $w_n$ be the number of weakly graded $({\bf 3+1})$-avoiding posets on $n$ vertices.

\begin{theorem}\label{thm:poset asymptotics}    
We have $g_n \sim w_n \sim n! \cdot \psi_n$.             
\end{theorem}   
The proof relies on the following special case of a theorem of Bender \cite{bender}, which may also be found in \cite[Theorem 7.3]{odlyzko}:
\begin{theorem}[{\cite[Theorem 1]{bender}}]\label{thm:bender}  
Suppose $A(x) = \sum_{n \geq 1} a_n x^n$, that $F(x, y)$ is a formal power series in $x$ and $y$, and that $B(x) = \sum_{n \geq 0} b_n x^n = F(x, A(x))$.  Let $C = \frac{\partial}{\partial y} F \Big|_{(0, 0)}$.  Suppose further that           
\begin{enumerate}          
\item $F(x, y)$ is analytic in a neighborhood of $(0, 0)$,     
\item $\displaystyle \lim_{n \to \infty} \frac{a_{n - 1}}{a_n} = 0$, and      
\item $\displaystyle \sum_{k = 1}^{n - 1} |a_k a_{n - k}| = O(a_{n - 1}).$      
\end{enumerate} 
Then            
\[              
b_n = C \cdot a_n + O(a_{n - 1})          
\]              
and in particular $b_n \sim C a_n$.       
\end{theorem}   
We now use this result to prove Theorem \ref{thm:poset asymptotics}.          
\begin{proof}   
Define          
\[              
F_1(x, y) = 1 + \frac{e^{2x}(2e^x - 3) + e^x(e^x - 2)^2 (y + 1)}{e^x(2e^x + 1) + (e^{2x} - 2e^x - 1)(y + 1)}  
\]              
and             
\[              
F_2(x, y) = (e^{-x} - 1)(y + 1) + \frac{2e^{3x} + e^{2x}(e^{x}-2)(y + 1)}{e^x(2e^{x} + 1) + (e^{2x} - 2e^x - 1)(y + 1)}      
\]              
so that         
\[              
F_1(x, \Psi(x) - 1) = G(x) 
\]              
is the exponential generating function for strongly graded $({\bf 3+1})$-avoiding posets (compare Theorem~\ref{thm:strong final answer}) and
\[              
F_2(x, \Psi(x) - 1) = W(x)
\]
is the exponential generating function for weakly graded $({\bf 3+1})$-avoiding posets (compare Theorem~\ref{thm:weak gf}).  (We compose with $\Psi - 1$ instead of $\Psi$ so that the composition is formally valid.)  We seek to apply Theorem~\ref{thm:bender} to compute asymptotics for the coefficients of $G$ and $W$.  To apply the theorem, we have three conditions to check.  The first condition is that $F_1$ and $F_2$ are analytic in a neighborhood of $(0, 0)$, which is clear by inspection. The second condition follows immediately from Proposition~\ref{prop:asymptotics for psi}.   The third condition is slightly trickier: we must show that       
\[              
\sum_{k = 1}^{n - 1} \psi_k \psi_{n - k} = O(\psi_{n - 1}).    
\]              
We do that now. 
 
The proof of Proposition~\ref{prop:asymptotics for psi} not only gives asymptotics for $\psi_n$ but also shows that          
\[              
\psi_n \leq C \frac{2^{n^2/4}}{\lfloor n/2\rfloor! \cdot \lceil n/2 \rceil!}  
\]              
for all $n$.  In addition, by taking only one term of the sum we have         
\[        
\psi_{n - 1} \geq \frac{2^{\lfloor (n - 1)/2\rfloor \cdot \lceil (n - 1)/2\rceil}}{\lfloor (n - 1)/2\rfloor! \cdot \lceil (n - 1)/2\rceil!} 
\]              
for all $n$.  Thus         
\[              
\sum_{k = 1}^{n - 1} \frac{\psi_k \psi_{n - k}}{\psi_{n - 1}} \leq C' \sum_{k = 1}^{n - 1}\frac{\lfloor (n - 1)/2\rfloor!  \cdot \lceil (n - 1)/2\rceil! \cdot  2^{- (k - 1)(n - k - 1)/2}}{\lfloor k/2\rfloor! \cdot \lceil k/2 \rceil! \cdot \lfloor (n - k)/2\rfloor! \cdot \lceil (n - k)/2 \rceil!}         
\]       
for some constant $C'$. This last summation is bounded by an absolute constant independent of $n$: the $k = 1$ and $k = n - 1$ terms of the sum are constant in $n$ while the $k = 2$ and $k = n - 2$ terms contribute a combined             
\[              
O\left(\lceil (n - 1)/2 \rceil 2^{-n/2}\right) = o(1)         
\]              
to the sum.  Each of the remaining terms can be seen to be $O(n^{-2})$, so their total contribution is also $o(1)$. 

We have shown that the conditions of Theorem~\ref{thm:bender} hold and we now apply it directly.  By direct computation, $\frac{\partial}{\partial y} F_1(x, y)\Big|_{(0, 0)} = \frac{\partial}{\partial y} F_2(x, y)\Big|_{(0, 0)} = 1$. Thus, we have
\[     
\frac{g_n}{n!} = \psi_n + O(\psi_{n - 1}) \sim \psi_n 
\]     
and similarly for $w_n$, as desired.
\end{proof}

\section*{Acknowledgments}
We wish to thank Alejandro Morales and Richard Stanley for valuable conversations.  We would also like to thank the two referees for many valuable comments, and in particular for correcting a significant error in Section~\ref{sec:sticking and gluing}.  YXZ was supported by an NSF graduate research fellowship.  JBL was supported in part by NSF RTG grant NSF/DMS-1148634.

\appendix

\section{Computing generating functions for quarks}
\label{section:yan calculus}

In this appendix, we enumerate and compute generating functions counting those sets of the form $A^\nu_\mu$ (introduced in Section~\ref{sec:gf for quarks}) that are of use to us.  For bookkeeping purposes, we make these generating functions bivariate in variables $x$ and $y$, with each bipartite graph in $A^\nu_\mu(m, n)$ (i.e., each graph on the vertex set $[m] \uplus [n]$ with appropriate restrictions) giving a contribution of $\frac{x^my^n}{m!n!}$.

It is very convenient to introduce the generating function
\[
\Psi(x, y) = \sum_{m,n \geq 0} \frac{2^{mn}x^my^n}{m!n!},
\]
as most of our generating functions are most easily expressed in terms of $\Psi(x,y)$.

\begin{enumerate}
\item $|A(m,n)| = 2^{mn}$: we have no restrictions, so all of the $mn$ edges may choose independently to be present or absent.  Equivalently, we have $\sum_{m, n \geq 1} |A(m, n)| \frac{x^m y^n}{m! n!} = \Psi(x,y) - e^x - e^y + 1$.  (The extra terms at the end simply account for the fact that we sum here only over positive values of $m$, $n$.)

\item $|A_\boxx(m,n)| = |A_\ox(m,n)| = (2^n-1)^m$: we need every vertex on the $m$-side to be not all-seeing (respectively, isolated) and there are no other restrictions.  It is not hard to compute the generating function
\[
\sum_{m, n \geq 1} |A_\ox(m, n)| \frac{x^m y^n}{m! n!} = e^{-x} \Psi(x,y) - e^y.
\]
It follows by symmetry that the generating function for $A^\boxx$ and $A^{\ox}$ is $e^{-y} \Psi(x, y) - e^x$.

\item $|A_{\boxx \ox}(m,n)| = (2^n-2)^m$: each vertex on the $m$-side can be connected to any subset on the $n$-side except the empty set or everything.  The associated generating function is $e^{-2x} \Psi(x,y) - e^{-x} - e^y + 1$.  

\item $|A^\boxx_\boxx(m,n)| = |B(m, n)|$: First, we show that 
\begin{equation}\label{eq:gf for B}
|B(m,n)| = \sum_{i = 0}^m (-1)^i \binom{m}{i} (2^{m-i}-1)^n.
\end{equation}
The proof is by inclusion-exclusion on the all-seeing vertices in $[n]$.  For a subset $S \subseteq [n]$, the number of graphs in which the vertices of $S$ are all-seeing and no vertices of $[m]$ are all-seeing is $(2^{n - |S|} - 1)^m$: each vertex in $[m]$ may choose the union of $S$ with any proper subset of $[n] \setminus S$ to be its neighbors, and these choices may be made independently.  Applying inclusion-exclusion immediately gives the result.  (As an aside, this means that the summation expression on the right-hand side of Equation~\eqref{eq:gf for B} is symmetric in $m$ and $n$, a fact not immediately obvious from its formula.)

Now, a routine calculation gives
\[
1 + \sum_{m,n \geq 1} |B(m,n)| \frac{x^m y^n}{m!n!} = \sum_{m,n \geq 0} |B(m,n)| \frac{x^m y^n}{m!n!} = e^{-x - y} \Psi(x, y).
\]

\item $|A_{\boxx \oo}^\boxx(m,n)| = |B_\oo(m, n)|$: From the definitions of the sets $A^\nu_\mu$ and the preceding computations we have  
\[
|A_{\boxx \oo}^\boxx| = |A_{\boxx \oo}| - |A_{\boxx \oo}^\boxo| = |A_{\boxx \oo}| = (2^n-1)^m - (2^n-2)^m.
\]
The associated generating function is $(1-e^{-x}) (e^{-x} \Psi(x,y) - 1)$.

\item $|A^{\boxx \oo}_{\boxx \oo}(m,n)| = |B^\oo_\oo(m,n)|$:  We have, by similar computations, that
\[
|A^{\boxx \oo}_{\boxx \oo}| = |A| - |A^{\ox}| - |A_\ox| + |B|
\]
and so the associated generating function is $
(1-e^{-x})(1-e^{-y})\Psi(x,y)$.
\end{enumerate}

Finally, we may use the work above to compute the generating functions we desire.  We have
\begin{align*}
F^\oo_\oo(x) & = \sum_{m, n \geq 1} |B^\oo_\oo(m, n)| \frac{x^{m+n}}{m!n!} \\
             & = (1 - e^{-x})^2 \Psi(x, x), \\[12pt]
F^\ox_\oo(x) & = \sum_{m, n \geq 1} |B^\ox_\oo(m, n)| \frac{x^{m+n}}{m!n!} \\
             & = \sum_{m, n \geq 1} \left(|B_\oo(m, n)| - |B^\oo_\oo(m, n)|\right) \frac{x^{m+n}}{m!n!} \\
             & = (1 - e^{-x})((2 e^{-x} - 1) \Psi(x, x) - 1),
\end{align*}
and similarly
\[
F^\oo_\ox(x) = (1 - e^{-x})((2 e^{-x} - 1) \Psi(x, x) - 1)
\]
and
\[
F^\ox_\ox(x) = (2e^{-x} - 1)((2e^{-x} - 1)\Psi(x, x) - 1).
\]

\bibliographystyle{alpha}
\bibliography{main}

\end{document}